\documentclass[10pt, a4paper]{amsart}
\usepackage{amsfonts}
\usepackage{amsthm}
\usepackage{amsmath}

\usepackage{amssymb}
\usepackage{latexsym}

\theoremstyle{plain} \newtheorem{lem}{Lemma}[section]
\theoremstyle{plain} \newtheorem{prop}[lem]{Proposition}
\theoremstyle{plain} 
\theoremstyle{plain} \newtheorem{cor}[lem]{Corollary}
\theoremstyle{plain} 

\theoremstyle{definition}

\newtheorem{rem}[lem]{Remark}

\newcommand{\Q}{\mathbb{Q}}
\newcommand{\R}{\mathbb{R}}
\newcommand{\C}{\mathbb{C}}

\newcommand{\Z}{\mathbb{Z}}
\newcommand{\Hind}{\mathrm{\mathop{Hind}}}
\newcommand{\ord}{\mathrm{\mathop{ord}}}
\newcommand{\Ord}{\mathcal{O}}

\newcommand{\Gal}{\mathrm{\mathop{Gal}}}
\newcommand{\diag}{\mathrm{\mathop{diag}}}
\newcommand{\dd}{\mathrm{\mathop{d}}}
\newcommand{\GL}{\mathrm{\mathop{GL}}}

\newcommand{\g}{\mathfrak{g}}
\newcommand{\h}{\mathfrak{h}}
\newcommand{\gl}{\mathfrak{gl}}

\begin{document}

\title{Computing generators of the unit group of an integral abelian group ring}
\author{Paolo Faccin}
\address{Dipartimento di Matematica\\
Universit\`{a} di Trento\\
Italy}
\email{faccin@science.unitn.it}
\author{Willem A. de Graaf}
\address{Dipartimento di Matematica\\
Universit\`{a} di Trento\\
Italy}
\email{degraaf@science.unitn.it}
\author{Wilhelm Plesken}
\address{Lehrstuhl B f\"ur Mathematik\\
RWTH Aachen University\\
Germany}
\email{plesken@momo.math.rwth-aachen.de}
\date{}

\begin{abstract}
We describe an algorithm for obtaining generators of the unit group of the integral
group ring $\Z G$ of a finite abelian group $G$. We used our implementation in {\sc Magma} 
of this algorithm to compute the unit groups of $\Z G$ for $G$ of order up to 110.
In particular for those cases we obtained the index of the group
of Hoechsmann units in the full unit group. At the end of the paper we describe 
an algorithm for the more general problem of finding generators of an arithmetic
group corresponding to a diagonalizable algebraic group. 
\end{abstract}

\maketitle

\section{Introduction}

Let $G$ be a finite abelian group. For a ring $R$ we let $RG$ be the group ring
over $R$, consisting of sums $\sum_{g\in G} a_g g $, with $a_g\in R$. We take $R=\Z$ and
consider the unit group
$$(\Z G)^* = \{ u\in \Z G \mid \text{there is a } v\in \Z G\text{ with } vu = 1\}.$$
Higman (\cite{higman}) showed that $(\Z G)^* = \pm G \times F$, where $F$ is a free abelian 
group. Moreover, Ayoub and Ayoub (\cite{ayoub}) have established that the rank of $F$ is 
$\tfrac{1}{2}(|G|+1+t_2-2l)$, where $t_2$ is the number of elements of $G$ of order 2, and 
$l$ is the number of cyclic subgroups of $G$.

Hoechsmann (\cite{hoechs}) described a construction of a set of generators of a 
finite-index subgroup of $(\Z G)^*$, called the group of constructable units.
Regarding this construction he wrote ``Does this method
ever yield all units if $n=|G|$ is not a prime power? The answer seems to be affirmative
for all $n<74$.''  In \cite{hoechs} this question is
not dealt with any further. When $n=74$ it is known that the group of constructable
units is of index 3 in the full unit group (see \cite{hoechs2}).
So the question remains whether the constructable units generate 
the full unit group if $|G| < 74$. 

In this paper we develop algorithms for computing generators of the unit group
$(\Z G)^*$. Using our implementation of these algorithms in the computer algebra
system {\sc Magma} (\cite{magma})
we have computed the unit groups for all abelian
groups of order $\leq 110$. We found 12 groups $G$ of order less than 74 whose unit
group is not generated by the Hoechsmann units (namely, the groups of order 40, 48, 
60, 63 and 65). 

In the next section we start by collecting some well-known facts and immediate observations
concerning lattices, groups, and associative algebras. Also, in the second half of the section
(Section \ref{sec:cyc}) we describe our approach to computing the unit group of the maximal
order in a cyclotomic field. This is achieved by combining a construction by Greither
(\cite{greither}) of a finite-index subgroup of the unit group, along with a {\sc Magma}
program by Fieker for ``saturating'' a subgroup at a given prime $p$. The latter algorithm
and its implementation will be described elsewhere.

Section \ref{sec:toralunit} contains the main algorithm of this paper, namely an
algorithm for computing the unit group of an order $\Ord$ in a toral algebra $A$.
The main idea is to split $A$ in its simple ideals $e_iA$, where the $e_i$ are 
orthogonal primitive idempotents. The $e_iA$ are number fields with orders $e_i\Ord$.
So in order to compute their unit groups we can use the effective version of the 
Dirichlet unit theorem (cf. \cite{cohen1}, \cite{pohstzas}). The basic step of the
algorithm is, given two orthogonal idempotents $\epsilon_1$, $\epsilon_2$,
to obtain the unit group of $(\epsilon_1+\epsilon_2)\Ord$ given the unit groups
of $\epsilon_i\Ord$, $i=1,2$.

In Section \ref{sec:groupunits} we describe our method for obtaining generators of 
unit groups of integral abelian group rings. Its main ingredients are the construction of 
the unit groups of cyclotomic fields, and the algorithm of Section \ref{sec:toralunit}.
We comment on the running times of the implementation of the algorithm in {\sc Magma},
and we give a table containing all abelian groups of orders up to 110, where the
constructable (or Hoechsmann-) units do not generate the full unit group. 
For all these groups we give the index of the group of constructable units in the 
full unit group.

Finally in the last section we indicate an algorithm to obtain generators of the
arithmetic group corresponding to a connected diagonalizable algebraic group defined
over $\Q$. Again the main ingredient is the algorithm of Section \ref{sec:toralunit}.

In our main algorithms and implementations we make essential use of Fieker's implementation
in {\sc Magma} of an algorithm by Ge (\cite{ge_thesis}) to obtain a basis of the lattice
$$\{ (\alpha_1,\ldots,\alpha_n)\in \Z^n \mid u_1^{\alpha_1}\cdots u_n^{\alpha_n} = 1\}$$
of multiplicative relations of given elements $u_1,\ldots,u_n$ in a number field. 

{\bf Acknowledgement:} we are very grateful to Claus Fieker for giving us early access to his
{\sc Magma} programs {\sf MultiplicativeGroup} and {\sf Saturation}, without which
this project could not have been carried out.

\section{Preliminaries}

\subsection{Lattices}\label{sec:lat}

In this paper we use the term ``lattice'' for a finitely generated subgroup of $\Z^m$. 
A lattice $\Lambda\subset \Z^m$ has a basis, that is a subset
$u_1,\ldots,u_r$ such that every $u\in \Lambda$ can uniquely
be written as $u = \sum_{i=1}^r \alpha_i u_i$, with $\alpha_i\in \Z$.
The lattice $\Lambda\subset\Z^m$ is called {\em pure} if $\Z^m/\Lambda$ is torsion-free.
(See \cite{cure}, \S III.16.) 

Let $\Lambda\subset \Z^m$ be a lattice with basis $u_1,\ldots,u_r$.
We form the $r\times m$-matrix $B$ with rows consisting of the coefficients 
of the $u_i$ with respect to the standard basis of $\Z^m$. By computing the 
the Smith normal form of $B$ we can effectively compute the homomorphism
$\psi : \Z^m \to \Z^m/\Lambda$ (cf. \cite{sims}, \S 8.3). Let $T$ denote 
the torsion submodule of $\Z^m/\Lambda$. Then $\psi^{-1}(T)$ is the smallest pure lattice
containing $\Lambda$. So in particular,
the Smith normal form algorithm gives a method to compute
a basis of the lattice $V\cap \Z^m$, where $V$ is a subspace of $\Q^m$. For example,
we can compute the intersection of lattices this way.

As observed in \cite{cure}, \S III.16, a lattice $\Lambda$ is pure if and only if
it is a direct summand of $\Z^m$. So in that case, by computing a Smith normal form,
we can compute a basis of $\Z^m$ such that the first $r$ basis elements form a basis
of $\Lambda$.

\begin{rem}
Computing a basis of the lattice $V\cap \Z^m$, where $V$ is a subspace of $\Q^m$, is called
{\em saturation}. There are methods known for this, based on computing a Hermite normal form, 
that are more efficient than the approach outlined above. Going into these matters would
lead us too far from the subject of this paper.
\end{rem}

\subsection{Toral algebras}\label{sec:toral}

We say that an associative algebra $A$ over $\Q$ is {\em toral} if it is 
semisimple, abelian and has an identity element, which we will denote by $e$. 
For example the group algebra $\Q G$ of a finite abelian group $G$ is toral. 

By the Wedderburn structure theorem (cf. \cite{rsp}, \S 3.5)
a toral algebra $A$ is a direct sum $A=A_1\oplus \cdots \oplus A_s$, where the $A_i$
are ideals that are isomorphic (as associative algebras) to field extensions of $\Q$.
A nonzero element $e_0\in A$ is said to be an idempotent if $e_0^2=e_0$. 
Two idempotents $e_1,e_2$ are
called orthogonal if $e_1e_2=0$. Furthermore, an idempotent is called primitive if it is
not the sum of orthogonal idempotents. Now the decomposition of $A$ into a direct
sum of simple ideals corresponds to a decomposition of the identity element $e\in A$ 
as a sum of primitive orthogonal idempotents, $e=e_1+\cdots +e_s$. Here $e_i$ is the
identity element of $A_i$, and vice versa, $A_i = e_iA$. We remark that there are 
algorithms to compute the $e_i$, given a basis of $A$ (cf. \cite{eber2}, \cite{eber}).

When $A=\Q G$, with $G$ a finite abelian group, there is a very efficient way to 
compute the primitive idempotents. 
Let $\chi : G\to \C^*$ be an irreducible character of $G$. Then
$$e_\chi = \frac{1}{|G|} \sum_{g\in G} \chi(g^{-1}) g$$
is an idempotent in $\C G$. Moreover, the $e_\chi$, as $\chi$ runs over all irreducible 
characters, are primitive orthogonal idempotents with sum $e$ (cf. \cite{cure}, Theorem 33.8). 
In particular, they form a
basis of $\C G$. Let $m$ denote the exponent of $G$, then all irreducible characters
$\chi$ have values in the cyclotomic field $\Q(\zeta_m)$. So the Galois group
$\Gal( \Q(\zeta_m) /\Q) \cong (\Z/m \Z)^*$ acts on the irreducible characters. Now we sum the
$e_\chi$, for $\chi$ in an orbit of $\Gal( \Q(\zeta_m) /\Q)$, and obtain the primitive 
orthogonal idempotents of $\Q G$.

A subset $\Ord$ of a toral algebra $A$, containing the identity of $A$, is said to be
an {\em order} (or, more precisely, a $\Z$-order) if there is a basis $a_1,\ldots,a_m$ 
of $A$ such that $\Ord = \Z a_1 +\cdots + \Z a_m$ and $a_ia_j \in \Ord$ for $1\leq i,j\leq m$.
For example, $\Z G$ is an order in $\Q G$. The unit group of $\Ord$ is
$$\Ord^* = \{ a\in \Ord \mid \text{ there is } b\in \Ord \text{ with } ab=e\}.$$
We consider the problem of obtaining a basis of the lattice of multiplicative relations
$$L = \{ (\alpha_1,\ldots,\alpha_r)\in \Z^r \mid a_1^{\alpha_1}\cdots a_r^{\alpha_r} = e \},$$
where $a_1,\ldots,a_r$ are given elements of $\Ord^*$. 
Note that $e_i A$ are number fields. So using Ge's algorithm (\cite{ge_thesis}) we can 
compute bases of the lattices
$$L_j = \{ (\alpha_1,\ldots,\alpha_r)\in \Z^r \mid (e_ja_1)^{\alpha_1}\cdots (e_ja_r)^{\alpha_r} = e_j 
\}.$$
Moreover, $L = \cap_j L_j$ so we can compute a basis of $L$ (see Section \ref{sec:lat}).

\subsection{Standard generating sets}\label{sec:cang}

Let $U$ be a finitely-generated abelian group. We say that a set of generators 
$g_1,\ldots,g_r$ of $U$ is {\em standard} if 
\begin{enumerate}
\item for $1\leq i\leq s$ the order of $g_i$ is $d_i$,
\item for $s+1\leq i\leq r$ the order of $g_i$ is infinite,
\item $d_i | d_{i+1}$ for $i<s$,
\end{enumerate}
and there are no other relations. So a standard set of generators immediately gives an isomorphism
of $U$ to $\Z/d_1\Z \oplus \cdots \oplus\Z/d_s\Z \oplus \Z^{r-s}$.

Giving finitely-generated abelian groups by standard generating sets yields straightforward
algorithms for several computational tasks concerning these groups, such
as computing the index of a subgroup, and computing the kernel of a homomorphism.
 
If an abelian group is given by a non-standard set of generators, then we can compute a 
standard one by computing the lattice of all relations of the generators, followed by
a Smith normal form computation (cf. \cite{sims}, \S 8.3). 
So, using the algorithm indicated in the previous section,
we can compute a standard set of generators for a finitely-generated subgroup of $\Ord^*$,
where $\Ord$ is an order in a toral algebra.

\begin{rem}
For many computational problems regarding finitely-generated abelian groups it suffices to
compute a Hermite normal form of the relation lattice. However, in our applications the
main computational problem is to obtain the relation lattice (see Table \ref{tab:0}), 
the subsequent computation
of the Smith normal form does not bear heavily on the running time. Therefore, for our purposes,
a Smith normal form is the most convenient. 
\end{rem}

\subsection{Units of cyclotomic fields}\label{sec:cyc}

Let $n$ be a positive integer. We consider the cyclotomic
field $\Q(\zeta_n)$, where $\zeta_n$ is a primitive $n$-th root of unity. The ring of integers
of this field is $\Z[\zeta_n]$. By Dirichlet's unit theorem the unit group $\Z[\zeta_n]^*$
is equal to $T\times F$, where $F$ is a free abelian group of rank $\tfrac{1}{2}\varphi(n)-1$,
and $T$ is the group of roots of unity in $\Q(\zeta_n)$. 
The problem considered in this section is to obtain generators of the unit group,
$\Z[\zeta_n]^*$. There are algorithms for this that work for any number field 
(cf. \cite{cohen1}, \cite{pohstzas});
but their complexity is such that it is only practical to use them for $n$ up to about 20
(depending on the hardware one uses, of course). For this reason we sketch a different
approach, using several results from the literature. The situation is straightforward
when $n$ is a prime power, see Section \ref{sec:cyc1}. Then in Section \ref{sec:cyc2}
we describe what can be done when $n$ is not a prime power. Using these methods 
we obtained a list of the generators of the unit groups $\Z[\zeta_n]^*$, for $n<130$.
However, for several $n$ the correctness of this list depends on the Generalised
Riemann Hypothesis, that is for $n$ prime between $67$ and $127$, and for $n=
115,119,121,123,125,129$ (so 19 cases in total).

Throughout we set $\Q(\zeta_n)^+ = \R\cap \Q(\zeta_n)$; 
then $\Q(\zeta_n)^+ = \Q(\zeta_n+\zeta_n^{-1})$. 
By $h_n^+$ we denote the class number of $\Q(\zeta_n)^+$.

\subsubsection{When $n$ is a prime power}\label{sec:cyc1}

Suppose that $n=p^m$ is a prime power. For $1< a<\tfrac{n}{2}$ with $\gcd(a,p)=1$ set
$$\xi_a = \zeta_n^{\tfrac{1-a}{2}} \frac{1-\zeta_n^a}{1-\zeta_n}.$$
Then $\xi_a$ lies in the unit group of $\Q(\zeta_n)^+$. Let $U_n$ be the group generated by
$-1$, $\zeta_n$ and all $\xi_a$. Then for the index we have $[\Z[\zeta_n]^* : U_n] =h_n^+$
(this is obtained by combining Corollary 4.13, Lemma 8.1 and Theorem 8.2 
in \cite{washington}). It is known that
$h_n^+=1$ if $\varphi(n) < 66$, and assuming the Generalised Riemann Hypothesis, we 
have $h_n^+=1$ when $\varphi(n)<162$ (see the Appendix in \cite{washington}). 
So for those $n$ we have generators of the unit group.

\subsubsection{When $n$ is not a prime power}\label{sec:cyc2}

Here the situation is more difficult. First of all we assume that $n\neq 2\bmod 4$, 
as for $n= 2\bmod 4$ we have that $\Q(\zeta_n)$ and $\Q(\zeta_{\tfrac{n}{2}})$ are isomorphic.
By $E^+$ denote the unit group of $\Q(\zeta_n)^+$. 
We use a finite-index subgroup of $E^+$ defined by Greither (\cite{greither}).
Here we briefly describe his construction. 

Let $\mathcal{G} = \Gal( \Q(\zeta_n)/\Q)$, and write the elements of $\mathcal{G}$ as
$\sigma_a$, where $\gcd(a,n)=1$ and $\sigma_a(\zeta_n) = \zeta_n^a$. For $\alpha =
\sum_a m_a \sigma_a\in \Z \mathcal{G}$ and $x\in \Q(\zeta_n)$ define
$$ x^\alpha = \prod_a \sigma_a(x)^{m_a}.$$
Let $n=\prod_{i=1}^s p_i^{e_i}$ be the factorisation of $n$ in distinct prime powers. Set
$S =\{1,\ldots,s\}$ and $P_S = \{ I\subset S\mid I\neq S\}$. For $I\in P_S$ we set
$n_I = \prod_{i\in I} p_i^{e_i}$. 

We consider arbitrary maps
$\beta : S \to \Z\mathcal{G}$, which we extend to maps (denoted by the same symbol)
$\beta : P_S \to \Z\mathcal{G}$ by $\beta(\emptyset)=1$, $\beta(\{i\})=\beta(i)$, and
$\beta(I\cup J) = \beta(I)\beta(J)$ if $I\cap J = \emptyset$. 
Now let $z\in \Q(\zeta_n)$. For $I\in P_S$ set $z_I = 1-z^{n_I}$, 
and $z(\beta) = \prod_{I\in S} z_I^{\beta(I)}$. Set $t = -\sum_{I\in S} n_I \beta(I)\in 
\Z\mathcal{G}$. Then for $a$ with $1< a < \tfrac{n}{2}$ and $\gcd(a,n)=1$ we consider
$$\xi_a(\beta) = \zeta_n^{d_a} \frac{\sigma_a(z(\beta))}{z(\beta)}, \text{ where }
d_a = \frac{(1-\sigma_a)t}{2}.$$
(Note that for $n$ odd, $\zeta_n^{\tfrac{1}{2}}$ lies in $\Z[\zeta_n]$, whereas for $n$ even we
have $a$ odd and hence $\tfrac{1-a}{2}$ is an integer.)

Following Greither we describe a good choice for $\beta$. First a small piece of notation:
if $g$ is an element of order $m$ of a group, then we set $N_g = 1+g+\cdots +g^{m-1}$,
which lies in the corresponding integral group ring. Now consider an $i\in S$. Let
$\mathcal{G}_i$ denote the Galois group of $\Q(\zeta_{n/p_i^{e_i}})^+$ over $\Q$. This group contains
the Frobenius automorphism $F_i$ (by definition: $F_i(\zeta_{n/p_i^{e_i}}) = \zeta_{n/p_i^{e_i}}^{p_i}$).
This yields the element $N_{F_i}$ in $\Z\mathcal{G}_i$. Now we define $\beta(i)$ to be 
a lift of $N_{F_i}$ to $\Z \mathcal{G}$. 

Let $C_\beta$ be the subgroup of $E^+$ generated by $-1$ and the $\xi_a(\beta)$.
Greither proved that $C_\beta$ does not depend on the choice for the lifts of the 
$N_{F_i}$, and that it is of index $h_n^+ i_\beta$ in $E^+$, where 
$$i_\beta = \prod_{i=1}^s e_i^{g_i-1}f_i^{2g_i-1};$$
here $e_i$, $f_i$ and $g_i$ are respectively the ramification, inertial and decomposition
degree of $p_i$ in $\Q(\zeta_n)^+$ (so that $e_if_ig_i = \tfrac{1}{2}\varphi(n)$).

Now let $U_n$ denote the group generated by $\zeta_n$ and $C_\beta$. Then using 
\cite{washington}, Corollary 4.13, we get that $[\Z[\zeta_n]^* : U_n] = 2h_n^+i_\beta$.

We used a program that Claus Fieker has written in {\sc Magma}, V2.17-2 (see
\cite{bifie} for some of the background). This program, given 
a finite-index subgroup $V$ of $\Z[\zeta_n]^*$, and a prime $p$, computes a subgroup
$\widetilde{V}$ of $\Z[\zeta_n]^*$, containing $V$, and such that the index $[\Z[\zeta_n]^*:
\widetilde{V}]$ is not divisible by $p$. We used this starting with the group $U_n$.
For $n< 130$ and $\varphi(n)\leq 72$ we have $h_n^+=1$, and assuming the Generalised
Riemann Hypothesis, we have $h_n^+=1$ for all $n< 130$ (\cite{washington}, Appendix).
Therefore we can compute $[\Z[\zeta_n]^*:U_n]$, and get 
all primes dividing it. So in the end we arrived at the full unit group $\Z[\zeta_n]^*$ for
all $n < 130$.

\section{Unit groups of orders in toral matrix algebras}\label{sec:toralunit}

Let $A$ be a toral algebra with identity $e$, and $\Ord\subset A$ an order.
In this section we consider the problem of computing a set of generators of $\Ord^*$.
First we consider some special cases.

\subsection{A simple toral algebra}\label{sec:1ip}

Let $A$ be a simple toral algebra. In other words, it is isomorphic
to a finite extension of $\Q$. Now there are algorithms for computing generators of an order 
in a number field (the effective version
of the Dirichlet unit theorem, see \cite{cohen1}, \cite{pohstzas}). We use these to compute
generators of $\Ord^*$ as well. 

\subsection{Two idempotents}\label{sec:twoidp}

Let $A$ be a toral algebra with identity $e$, 
and $e_1,e_2\in A$ orthogonal (but not necessarily
primitive) idempotents with $e_1+e_2 = e$. Set $A_i = e_iA$, then $A=A_1\oplus A_2$.
Let $\Ord$ be an order in $A$, then $\Ord_i= e_i\Ord$ is an order in $A_i$. Here we suppose
that we have generators of $\Ord_i^*$, $i=1,2$, and the problem is to find generators of
$\Ord^*$. 

Set $J= (e_1\Ord \cap \Ord) + (e_2\Ord\cap \Ord)$. Since this is an ideal in $\Ord$ we can
form the quotient $R = \Ord/J$. Consider the maps $\varphi_i : e_i \Ord \to R$ defined
by $\varphi_i(e_i a) = a + J$, for $a\in \Ord$. (Note that this is well-defined: if $e_1a=e_1b$
then $a-b=e_2(a-b)$ so it lies in $J$.) These are surjective ring homomorphisms
with respective kernels $e_i\Ord\cap\Ord$.

\begin{lem}
$\Ord = \{ a_1 +a_2 \mid a_i \in e_i\Ord \text{ and } \varphi_1(a_1) = \varphi_2(a_2)\}.$
\end{lem}

\begin{proof}
Let $a\in \Ord$ and set $a_i = e_ia$ then $a=a_1+a_2$ and $a_i\in e_i\Ord$. Moreover
$\varphi_1(a_1) = a+J = \varphi_2(a_2)$. Conversely, let $a,b\in \Ord$ be such that
$\varphi_1(a_1) = \varphi_2(a_2)$, where $a_1=e_1a$ and  $a_2 =e_2b$. Then $a-b\in J$, whence
$a=b+u_1+u_2$ where $u_i \in e_i\Ord \cap \Ord$. Therefore $e_1a = e_1 b+u_1$ so that
$a_1+a_2=e_1a+e_2b=e_1b+e_2b+u_1=b+u_1$ which lies in $\Ord$. 
\end{proof}

\begin{cor}
$\Ord^* = \{ a_1 +a_2 \mid a_i \in (e_i\Ord)^* \text{ and } \varphi_1(a_1) = \varphi_2(a_2)\}.$
\end{cor}

So in order to compute generators of $\Ord^*$ we perform the following steps:
\begin{enumerate}
\item Compute bases of $\Ord_i = e_i\Ord$.
\item Compute bases of $\Ord_i\cap \Ord$, of $J = (\Ord_1\cap \Ord) + (\Ord_2\cap \Ord)$,
and set $R=\Ord/J$.
\item Compute generators of the groups $H_i = \varphi_i(\Ord_i^*) \subset R^*$ and
$H = H_1\cap H_2$.
\item Compute generators of the groups $M_i = \varphi_i^{-1}(H)\subset \Ord_i^*$.
\item Compute generators of the group $\Ord^* = \{ a_1+a_2 \mid a_i\in M_i \text{ and }
\varphi_1(a_1)=\varphi_2(a_2)\}$. 
\end{enumerate}

\subsubsection{Implementation}

We comment on the implementation of the steps of the algorithm. 
Step (1) is done by a Hermite normal form computation. The intersections in Step
(2) are computed using the techniques indicated in Section \ref{sec:lat}. 
Note that a basis of $J$ is obtained
by concatenating the bases of $\Ord_i\cap \Ord$. The ring $R=\Ord/J$ can be constructed
by a Smith normal form computation.

For Step (3) we assume that $e_2 A$ is isomorphic to a number field (in other words, that
$e_2$ is a primitive idempotent). When using the algorithm, this can always be arranged
(see Section \ref{sec:genc}). 
Then $e_2 \Ord$ an order in it. We set $I = e_2\Ord \cap \Ord$ and
use the isomorphism $R\cong (e_2 \Ord) / I$. Subsequently we use algorithms described in 
\cite{hpp}, \cite{klupau} to compute a standard generating set of $(e_2 \Ord / I)^*$.
So computing $H_1$, $H_2$ as subgroups of $(e_2 \Ord / I)^*$ we can perform the operations
of Step (3).

In Step (4) we view $\varphi_i$ as a homomorphism $\Ord_i^*\to H_i$. We use this to compute
generators of the kernel of $\varphi_i$ as well as pre-images of the generators of $H$.
Together these generate the group $M_i$. As remarked in Section \ref{sec:cang}, we can compute
standard generating sets of the groups $\Ord_i^*$. Using these,
it is straightforward to obtain a standard generating
set for the subgroups $M_i$. 
Then we restrict $\varphi_i$ to obtain a homomorphism $\varphi_i : M_i\to H$.

Now we come to Step (5). Let $h_1,\ldots,h_r$, $a_1,\ldots,a_s$, $b_1,\ldots,b_t$ be standard
generating sets of $H$, $M_1$ and $M_2$ respectively. Set
$$\Lambda = \{ (\alpha_1,\ldots,\alpha_s,\beta_1,\ldots,\beta_t)\in \Z^{s+t} \mid 
\varphi_1(a_1^{\alpha_1}\cdots a_s^{\alpha_s}) = \varphi_2(b_1^{\beta_1}\cdots \beta_t^{e_t})\}.$$
Then 
$$\Ord^* = \{ a_1^{\alpha_1}\cdots a_s^{\alpha_s} + b_1^{\beta_1}\cdots b_t^{\beta_t} \mid 
(\alpha_1,\ldots,\alpha_s,\beta_1,\ldots,\beta_t)\in\Lambda\}.$$
Moreover, $\Lambda$ is a lattice, hence has a finite basis. Furthermore,
the elements of $\Ord^*$ corresponding to the elements of a basis of $\Lambda$ generate
$\Ord^*$. So the problem of finding a generating set of $\Ord^*$ is reduced to finding a basis
of $\Lambda$.

Define $\mu_{ij},\nu_{ij}\in \Z$ by 
\begin{align*}
\varphi_1(a_i) &= \prod_{j=1}^r h_j^{\mu_{ij}}\\
\varphi_2(b_i) &= \prod_{j=1}^r h_j^{\nu_{ij}}.
\end{align*}
A small calculation shows that $(\alpha_1,\ldots,\alpha_s,\beta_1,\ldots,\beta_t)\in \Lambda$ 
if and only if
\begin{equation}\label{eq:1}
\sum_{i=1}^s \mu_{ij}\alpha_i -\sum_{k=1}^t \nu_{kj}\beta_k = 0 \bmod \ord(h_j)
\end{equation}
for $1\leq j\leq r$ (and where $\ord(h_j)$ denotes the order of $h_j$).
Let $S$ be the integral matrix with columns
$$(\mu_{1j},\ldots,\mu_{sj},-\nu_{1j},\ldots,-\nu_{tj},\ord(h_j))$$
for $1\leq j\leq r$. We compute a basis the integral kernel of $S$, which is the set of
all $v\in \Z^{s+t+1}$ such that $vS=0$. For each $v$ in this basis we take the vector
consisting of the first $s+t$ coordinates. This way we obtain a basis of $\Lambda$.

\subsection{The general case}\label{sec:genc}

Now let $A$ be a toral algebra, $e_1,\ldots,e_m$ its primitive orthogonal 
idempotents with sum $e$, and $A_i = e_iA$ the corresponding simple ideals. Let $\Ord$ be an
order in $A$; then $e_i\Ord$ is an order in $A_i$, and as indicated in Section \ref{sec:1ip},
we can compute generators of the unit groups $(e_i\Ord)^*$. Then for $j=2,3,\ldots$ we set 
$\epsilon_j = e_1+\cdots +e_j$ and we apply the algorithm of Section 
\ref{sec:twoidp} to the algebra $\epsilon_jA$ with its order $\epsilon_j \Ord$,
and two idempotents $\epsilon_{j-1}$ and $e_j$, yielding the unit group $(\epsilon_j\Ord)^*$.
When the algorithm terminates we have the unit group $\Ord^*$.

\section{Units of integral abelian group rings}\label{sec:groupunits}

Let $G$ be an abelian group. 
As seen in Section \ref{sec:toral}, it is straightforward to compute primitive
orthogonal idempotents $e_1,\ldots,e_r\in \Q G$ such that
$e_i(\Q G)$ is isomorphic to a field extension of $\Q$. Let $m$ denote the exponent of $G$,
then (cf. \cite{ayoub})
$$ \Q G \cong \bigoplus_{d|m} \bigoplus_{i=1}^{t_d} \Q(\zeta_d),$$
where $\Q(\zeta_d)$ is the cyclotomic field of order $d$, and $t_d$ is the number of cyclic
subgroups of $G$ of order $d$. In particular, $e_i(\Q G)\cong \Q(\zeta_{d_i})$. 
We consider the order
$\Z G$ in $\Q G$. We have that $e_i(\Z G)$ is isomorphic to $\Z[\zeta_{d_i}]$,
and by the results of Section \ref{sec:cyc}, we have generators of $\Z[\zeta_{d_i}]^*$
for $d_i < 130$. 
So for small groups $G$ we can apply the algorithm of Section \ref{sec:toralunit}
to obtain generators of the unit group $(\Z G)^*$. Using our implementation of 
the algorithms in {\sc Magma}, we have carried this out for all abelian groups of
orders up to $110$. In Table \ref{tab:0} we collect some timings and other data
related to the algorithm.

\begin{table}[htb]\label{tab:0}
\begin{tabular}{|l|r|r|r|r|}
\hline
$G$ & $\varphi(\exp(G))$ & $|\mathrm{digits}|$ &
$t_m$ & $t_{\mathrm{tot}}$ \\
\hline
$C_{70}$ & 24 & 7.4 & 314 & 328 \\
$C_{80}$ & 32 & 63.4 & 1131 & 1183 \\
$C_{90}$ & 24 & 159.5 & 1043 & 1078\\
$C_{91}$ & 72 & 3.7 & 2352 & 2446 \\
$C_{96}$ & 32 & 31.2 & 2322 & 2373 \\
$C_2\times C_{48}$ & 16 & 181.3 & 1575 & 1781 \\
$C_2\times C_2\times C_{24}$ & 8 & 54.6 & 1031 & 1267\\
$C_2\times C_4\times C_{12}$ & 4 & 22.5 & 537 & 725 \\
$C_{100}$ & 40 & 217.6 & 3822 & 3942 \\
$C_2\times C_{50}$ & 20 & 66352.6 & 414569 & 426654\\
$C_5\times C_{20}$ & 8 & 379.7 & 1227 & 1425\\
$C_{10}\times C_{10}$ & 4 & 275.2 & 1131 & 1332\\
\hline
\end{tabular}
\caption{Runtimes (in seconds) for the algorithm to compute generators of $(\Z G)^*$. 
The first column lists the isomorphism type of $G$, and the second the value of the Euler 
$\varphi$-function on the exponent of $G$. The
third column has the average number of digits of the coefficients of the units output
by the algorithm (with respect to the standard basis of $\Z G$).
The
fourth column displays the time spent to compute multiplicative relations in cyclotomic fields. 
The last column has the total time spent by the algorithm.}
\end{table}

From Table \ref{tab:0} we see that the running time is dominated by the time needed to compute
multiplicative relations in cyclotomic fields (Ge's algorithm). This algorithm needs to work
harder if the degrees of the fields that occur are higher. Indeed, the running times generally 
increase when $\varphi(\exp(G))$ increases (note that this is the highest degree of a 
cyclotomic field occurring in the decomposition of $\Q G$). However, also the size of the
elements of which we need to compute multiplicative relations plays a role. For some groups
the average number of digits of a unit, as output by the algorithm, is very high. This is
seen most dramatically for $C_2\times C_{50}$. Note also that the size (i.e., the average
number of digits of their coefficients) of the units output by the algorithm is far
from being optimal; indeed, for $G=C_2\times C_{50}$, the unit group $(\Z G)^*$ is also
generated by the Hoechsmann units (see below).

Various constructions of finite index subgroups of $(\Z G)^*$ have appeared in the 
literature (see \cite{sehgal_units} for an overview). Among these a construction by
Hoechsmann (\cite{hoechs}) seems to yield subgroups of particularly small index.
As an application of our algorithm, we have compared the full unit groups with the 
groups of Hoechsmann units, for $G$ of size up to $110$.

We briefly describe Hoechsmann's construction.
First, let $C$ be a cyclic group of order $n$, generated
by the element $x$. For $i\geq 0$ and $y\in C$ we set
$$s_i (y) = 1+y+\cdots +y^{i-1}.$$
Let $i,j$ be integers with $0 < i,j < n$ and $\gcd(i,n)=\gcd(j,n)=1$. Let $k,l$ be positive
integers with $li = 1 +kn$. Set
$$u_{i,j}(x) = s_l(x^i)s_i(x^j)-ks_n(x).$$
Then $u_{i,j}(x)$ is a unit in $(\Z C)^*$. Let $\Theta(C)$ denote the set of all units
constructed in this way. We let $\mathcal{H}$ be the group of units in $(\Z G)^*$ generated by
all $\Theta(C)$, where $C$ ranges over the cyclic subgroups of $G$ of order $>2$,
along with $\pm G$. It is called the group of {\em constructable units} of $\Z G$.
By \cite{hoechs}, Theorem 2.5, $\mathcal{H}$ is a subgroup of finite index
of $(\Z G)^*$. 

In general the set consisting of the $u_{i,j}(x)$ is a heavily redundant generating set. 
We note, however, that \cite{hoechs2} Theorem 1.1, gives a non-redundant generating set
of $\mathcal{H}$, if the group $H_m = (\Z /m\Z)^*/\{\pm 1\}$ is cyclic, where $m$ is the
exponent of $G$. We remark that for small $m$ the group
$H_m$ is cyclic with only few exceptions. The values of $m$ up to 120, for which 
$H_m$ is not cyclic, are 24, 40, 48, 56, 60, 63, 65, 72, 80, 84, 85, 88, 91, 96, 104,
105, 112, 117, 120.

We let $\Hind(G)$ be the index of $\mathcal{H}$ in $(\Z G)^*$, and we call this number
the Hoechsmann index. We have computed the Hoechsmann indices for all abelian groups
of orders up to 110.
For most groups the index is 1. The groups for which it is not
1 are listed in Table \ref{tab:1}, along with the corresponding
Hoechsmann indices. 

From Table \ref{tab:1} we see that 
on many occasions when $\Hind(G)\neq 1$ we have that $|G|=m$, with $H_m$ not cyclic.
Also, for the groups considered,
if for one group $G$ of order $m$ we have $\Hind(G)\neq 1$, then the same holds
for all groups of that order (except when $(\Z G)^*$ has rank 0).

\begin{table}[htb]\label{tab:1}
\begin{tabular}{|l|c|l|c|}
\hline
$G$ & $\Hind(G)$ & $G$ & $\Hind(G)$ \\
\hline
$C_{40}$ & 2 & $C_{84}$ &  2 \\   
$C_{2}\times C_{20}$ & 2 & $C_2\times C_{42}$ & 2\\  
$C_2\times C_2\times C_{10}$ & 2 & $C_{85}$ & 2 \\  
$C_{48}$ & 2 & $C_{90}$ & 3 \\
$C_2\times C_{24}$ & 2 & $C_3\times C_{30}$ & 3\\
$C_2\times C_2\times C_{12}$ & 2 & $C_{91}$ & 3 \\
$C_4\times C_{12}$ & 4 & $C_{96}$ & 8\\
$C_{60}$ & 2 & $C_2\times C_{48}$ & 16\\  
$C_2\times C_{30}$ & 2 & $C_2\times C_2\times C_{24}$ & 32  \\
$C_{63}$ & 3 & $C_2\times C_2\times C_2\times C_{12}$ & 16 \\ 
$C_3\times C_{21}$ & 3 & $C_2\times C_4\times C_{12}$ & 64\\ 
$C_{65}$ & 2 & $C_4\times C_{24}$ & 64 \\ 
$C_{74}$  & 3 & $C_{98}$ & 7\\
$C_{80}$ & 4  & $C_7\times C_{14}$ & 343 \\  
$C_2\times C_{40}$ & 8 & $C_{104}$ & 2\\
$C_2\times C_2\times C_{20}$ &  16 & $C_2\times C_{52}$ & 2\\ 
$C_2\times C_2\times C_2\times C_{10}$ &  32 &  $C_2\times C_2\times C_{26}$ &  2\\
$C_4\times C_{20}$ &  8 & $C_{105}$ & 4 \\ 
\hline
\end{tabular}
\caption{Hoechsmann indices for abelian groups of orders up to 110, for which this
index is not 1.}
\end{table}

\begin{rem}
For the groups $C_p$, with $p$ a prime between 67 and 120 the
correctness of our computation depends on the Generalised Riemann Hypothesis (see
Section \ref{sec:cyc}).
\end{rem}

\section{Computing generators of arithmetic groups of diagonalisable algebraic groups}

In this section we use the term {\em D-group} as short for 
for diagonalisable algebraic subgroup of $\GL(n,\C)$, defined over $\Q$. 
We consider the problem to compute generators of the arithmetic group $G(\Z) = 
G\cap \GL(n,\Z)$ for a given connected D-group $G$. We assume that the group $G$ is 
given by its Lie algebra $\g\subset \gl(n,\C)$ which in turn is given by a basis
of matrices with coefficients in $\Q$. In the sequel such a basis will be called 
a {\em $\Q$-basis}.

We have two subsections: the first is devoted to characters, and in the second
we outline our algorithm.

\subsection{Computing the characters}\label{character}

We recall that a character of a group $H$ is a homomorphism $\chi : H\to F^*$, 
where $F^*$ denotes the multiplicative group of nonzero elements of a field $F$.
In this section we let $G\subset H\subset \GL_m(\C)$ be two connected 
D-groups defined
over $\Q$. A well-known theorem (cf. \cite{borel}, Proposition 8.2)
states that there are characters $\chi_1,\ldots,\chi_t : H\to \C^*$ such that
$G$ is the intersection of their kernels, i.e.,
$$G = \{ h\in H \mid \chi_i(h) = 1 \text{ for } 1\leq i\leq s\}.$$

The problem considered in this section is to find a number field $K$,
and characters $\chi_1,\ldots,\chi_t : H\to K^*$ such that $G$ is the
intersection of their kernels. We require that $K$ and the $\chi_i$ be
explicitly given, i.e., that we know the minimum polynomial of a primitive
element of $K$, and that for a given $h\in H$ we can effectively compute
$\chi_i(h)$. We assume that we are given the Lie algebras $\h,\g$ of $H$
and $G$ respectively, in terms of $\Q$-bases.

\begin{prop}\label{lem:char1}
Let $X\in \GL_m(\C)$ be such that $\widetilde{H}=XHX^{-1}$ consists of 
diagonal matrices. Set $\tilde\h = X\h X^{-1}$, $\tilde \g = X\g X^{-1}$. 
Then $\tilde\h, \tilde\g$ consist of diagonal matrices as well. Set
$$\Lambda(\tilde\h) = \{ (e_1,\ldots,e_m)\in \Z^m \mid e_1\alpha_1+\cdots
+e_m\alpha_m = 0 \text{ for all } \diag(\alpha_1,\ldots,\alpha_m)\in 
\tilde\h\},$$
and similarly define $\Lambda(\tilde\g)$. Then $\Lambda(\tilde\h) \subset
\Lambda(\tilde\g)\subset \Z^m$ are pure lattices 
(cf. Section \ref{sec:lat}). Let  $\underline e_i =(e_{i,1},\ldots,
e_{i,m})\in \Z^m$, for $1\leq i\leq s$, be a basis of $\Lambda(\tilde \g)$
such that $\underline e_{t+1},\ldots,\underline e_s$ is a basis of
$\Lambda(\tilde \h)$. 
For $1\leq i\leq t$ define
$\chi_i : H\to \C^*$ by $\chi_i(h) = \alpha_1^{e_{i,1}}\cdots \alpha_m^{e_{i,m}}$, 
where
$XhX^{-1} = \diag(\alpha_1,\ldots,\alpha_m)$. Then $G$, as subgroup of $H$, 
is the intersection of the kernels of the $\chi_i$.
\end{prop}

\begin{proof}
Note that $\tilde{\h}$ is the Lie algebra of $\widetilde{H}$. Therefore,
it consists of diagonal matrices. Since $\tilde\g \subset \tilde \h$
we get the same property for $\tilde\g$.
The lattices $\Lambda(\tilde\h)$, $\Lambda(\tilde\g)$ are
obviously pure. So as noted in Section \ref{sec:lat}, the required basis 
of $\Lambda(\tilde \g)$ exists.

A character $\widetilde\chi$ of $\widetilde{H}$ is given by a sequence 
$(e_1,\ldots,e_m)$ of integers such that $\widetilde\chi(\diag(\alpha_1,
\ldots,\alpha_m)) = \alpha_1^{e_1}\cdots \alpha_m^{e_m}$. 
This follows from the same statement 
for the full diagonal group $D_m$, along with the fact that a character of 
$\widetilde{H}$ extends to a character of $D_m$ (\cite{borel}, Proposition 8.2).
Furthermore, the differential of $\widetilde\chi$ is given by 
$\dd \widetilde\chi (\diag(\alpha_1,\ldots,\alpha_m)) = e_1\alpha_1+\cdots 
+e_m \alpha_m$. So since $\tilde \h$ consists of all $\diag(\alpha_1,\ldots,
\alpha_m)$ with $e_1\alpha_1+\cdots +e_m\alpha_m=0$ for all $(e_1,\ldots,e_m)
\in \Lambda(\tilde \h)$ we get that 
$$\widetilde{H} = \{ \diag(\alpha_1,\ldots,\alpha_m) \mid \alpha_1^{e_1}\cdots 
\alpha_m^{e_m}=1 \text{ for all } (e_1,\ldots,e_m)\in \Lambda(\tilde \h)\}.$$
Indeed, the latter is a connected algebraic group with the same Lie algebra
as $H$. An analogous statement holds for $\widetilde{G} = XGX^{-1}$ and
$\tilde\g$. This implies the statement of the proposition.
\end{proof}

So in order to solve the problem of this section  we need to say how the
various objects in Proposition \ref{lem:char1} can be computed. 

Let $B\subset M_m(\Q)$ be the associative algebra with one generated by
the basis elements of $\h$. 
Then, using the terminology of Section \ref{sec:toral}, $B$ is
toral (indeed: its elements can simultaneously be diagonalized). Now using algorithms
given in for example \cite{eber3}, \cite{grf_ivany} we can compute a {\em splitting element}
$b_0\in B$. This means that $b_0$ generates $B$ (as associative algebra).
By repeatedly factoring polynomials over number fields, we
construct the splitting field $K$ of the minimal
polynomial of $b_0$ over $\Q$. After factoring this polynomial over $K$ we
find the eigenvalues and eigenvectors of $b_0$. From this we get a matrix
$X\in \GL_m(K)$ such that $Xb_0X^{-1}$ is diagonal. We note that this also
implies that $X\h X^{-1}$ consists of diagonal matrices. As this is the
Lie algebra of $\widetilde{H} = XHX^{-1}$, we get the same for $\widetilde{H}$,
as the latter group is connected.

Since we have given a basis of $\h$ we can compute a basis of $X\h X^{-1}$.
Corresponding to a basis element $\diag(\alpha_1,\ldots,\alpha_m)$ 
we construct the linear equation $\alpha_1x_1+\cdots \alpha_mx_m=0$ 
in the unknowns $x_i$. We can express the $\alpha_i$ as vectors 
in $\Q^{d}$, where $d$ is the degree of $K$. So by solving these
equations we get a basis of the space 
$$V(\tilde\h) = \{ (a_1,\ldots,a_m)\in 
\Q^m\mid a_1\alpha_1+\cdots +a_m\alpha_m=0 \text{ for all } \diag(\alpha_1,
\ldots ,\alpha_m)\in \tilde\h\}. $$
Now $\Lambda(\tilde\h) = V(\tilde\h)\cap \Z^m$, so as seen in Section \ref{sec:lat},
we can compute a basis of $\Lambda(\tilde\h)$. Furthermore, using the methods
outlined in the same section, we get a basis of $\Lambda(\tilde\g)$ that contains a basis
of $\Lambda(\tilde\h)$. Then the characters $\chi_i$ are 
easily constructed as in Proposition \ref{lem:char1}.

\subsection{Putting the pieces together}\label{sec:together}

Now let $G\subset \GL_m(\C)$ be a connected D-group, defined over $\Q$.
We assume that we have given its Lie algebra $\g \subset \gl_m(\C)$ by a 
$\Q$-basis.
We describe an algorithm for obtaining a finite set of
generators of the group $G(\Z) = G\cap \GL_m(\Z)$. 

\begin{lem}\label{lem:fin1}
Let $A\subset M_m(\C)$ be the associative algebra with one generated by the
basis elements of $\g$. Let $A^*$ denote the set of
invertible elements of $A$. Then $A^*$ is a D-group defined over $\Q$ with
Lie algebra $A$. Furthermore, $G\subset A^*$.
\end{lem}

\begin{proof}
Note that $A$ has a basis consisting of matrices with coefficients in $\Q$.
Let $A(\Q)$ denote the $\Q$-algebra spanned by such a basis. Consider the 
linear equations that define $A(\Q)$ as a subspace of $M_m(\Q)$. Then those
equations define $A^*$ as a subgroup of $\GL_m(\C)$. Since the basis elements
of $\g$ can be simultaneously be diagonalised, the same holds for the 
elements of $A^*$. It follows that $A^*$ is a D-group, defined over $\Q$.

Since $A^*$, as subgroup of $\GL_m(\C)$, is given by linear equations, its
Lie algebra, as subalgebra of $\gl_m(\C)$, is given by the same equations.
Hence the Lie algebra of $A^*$ is $A$, where the Lie product is given by the
commutator. So since $\g \subset A$, we also have $G\subset A^*$, as $G$ 
is connected.
\end{proof}

Now in order to compute generators of $G(\Z)$ we have the following
algorithm. As input we take a $\Q$-basis of $\g$.

\begin{enumerate}
\item Compute a $\Q$-basis of $A$ (notation of Lemma \ref{lem:fin1}).
\item Let $\Ord = A\cap M_m(\Z)$
and use the algorithm of Section \ref{sec:toralunit}
to compute generators $a_1,\ldots,a_r$ of $\Ord^*$. 
\item Use the algorithm of Section \ref{character} to construct a 
number field $K$ and characters $\chi_1,\ldots,\chi_s : A^* \to K^*$
such that $G$ is given, as subgroup of $A^*$, by the intersection of
their kernels.
\item Use Ge's algorithm (\cite{ge_thesis}) to compute bases of the
lattices
$$\Lambda_i = \{ (e_1,\ldots,e_r)\in \Z^r \mid \chi_i(a_1)^{e_1}\cdots
\chi_i(a_r)^{e_r}=1\}.$$
\item Let $\underline{e}_k = (e_{k,1},\ldots,e_{k,r})$, for $1\leq k\leq t$,
be a basis of the intersection of all $\Lambda_i$ ($1\leq i\leq s$).
Set $u_k = a_1^{e_{k,1}}\cdots a_r^{e_{k,r}}$. Then $u_1,\ldots,u_t$ is a
set of generators of $G(\Z)$.
\end{enumerate}

\begin{prop}
The previous algorithm terminates correctly. 
\end{prop}

\begin{proof}
Note that by Lemma \ref{lem:fin1},
$A$ is the Lie algebra of $A^*$, so that in Step (3) 
we have the correct input for the algorithm of Section \ref{character}.
A basis of $\Ord$ can be computed using the methods of Section
\ref{sec:lat}. So all steps can effectively be carried out.

Observe that $A^*(\Z) =A^* \cap \GL_m(\Z)$ is equal to $\Ord^*$. 
Now $g\in G(\Z)$ if and only if $g\in A^*(\Z)$ and $\chi_i(g)=1$ for
$1\leq i\leq s$. But this is equivalent to $g = a_1^{e_1}\cdots a_r^{e_r}$ for
certain $e_i\in \Z$, and $\chi_i(g)=1$ for $1\leq i\leq s$. But the latter
condition is equivalent to $(e_1,\ldots,e_r)\in \Lambda_i$ for $1\leq i\leq s$.
Hence $g\in G(\Z)$ if and only if $g$ is a product of the $u_k$.
\end{proof}

\end{document}